\documentclass[12pt,a4paper,reqno]{amsart}
\usepackage[mathcal]{eucal}
\usepackage{amssymb}
\usepackage[nice]{nicefrac}
\usepackage{hyperref}
\usepackage[normalem]{ulem}
\usepackage{xcolor}

\theoremstyle{plain}
\newtheorem{theorem}{Theorem}[section]
\newtheorem{lemma}[theorem]{Lemma}
\newtheorem{proposition}[theorem]{Proposition}
\newtheorem{corollary}[theorem]{Corollary}

\theoremstyle{definition}

\newtheorem{remark}[theorem]{Remark}

\numberwithin{equation}{section}

\DeclareMathOperator{\hdim}{hdim}
\DeclareMathOperator{\hspec}{hspec}

\makeatletter
\newcommand{\AlignFootnote}[1]{%
    \ifmeasuring@
    \else
        \footnote{#1}%
    \fi
}
\makeatother

\newcommand{\Z}{\mathbb{Z}}

\newcommand{\N}{\mathbb{N}}

\title{A pro-$2$ Group with full normal Hausdorff Spectra}
\author[I. de las Heras]{Iker de las Heras} 
\address{Iker de las Heras: Mathematisches Institut, Heinrich-Heine-Universit\"at, 40225
  D\"usseldorf, Germany; Department of Mathematics, University
  of the Basque Country UPV/EHU, 48940 Leioa, Spain}
\email{iker.delasheras@hhu.de}

\author[A. Thillaisundaram]{Anitha Thillaisundaram} 
\address{Anitha Thillaisundaram: Centre for Mathematical Sciences, Lund University, 223 62 Lund, Sweden}
\email{anitha.t@cantab.net}

\date{\today}

\thanks{The first author is supported by the Spanish Government, grant MTM2017-86802-P, partly with FEDER funds, and by the Basque Government, grant IT974-16. He is also supported by a postdoctoral grant of the Basque Government. The second author acknowledges the support from EPSRC, grant EP/T005068/1. %Both authors thank the Heinrich-Heine-Universit\"{a}t D\"{u}sseldorf, where a large part of this research was carried out.
}



\begin{document}

\maketitle

\begin{abstract}
We construct a $2$-generated pro-$2$ group with full normal Hausdorff spectrum $[0,1]$, with respect to each of the four 
standard filtration series: the $2$-power series, the lower $2$-series, the Frattini series, and the dimension subgroup series. 
This answers a question of Klopsch and the second author, for the even prime case; 
the odd prime case was settled by the first author and Klopsch. Also, our construction gives the first example of a finitely generated pro-$2$ group with full Hausdorff spectrum with respect to the lower $2$-series.
\end{abstract}

\section{Introduction}

Let $\Gamma$ be a countably based infinite
profinite group and consider a \emph{filtration series} $\mathcal{S}$
of~$\Gamma$, that is, a descending chain
$\Gamma = \Gamma_0 \ge \Gamma_1 \ge \cdots$ of open normal subgroups
$\Gamma_i \trianglelefteq_\mathrm{o} \Gamma$ such that $\bigcap_i \Gamma_i = 1$.  These open normal subgroups form a base of neighbourhoods of the identity and
induce a translation-invariant metric on~$\Gamma$ given by
$d^\mathcal{S}(x,y) = \inf \left\{ \lvert \Gamma : \Gamma_i \rvert^{-1} \mid x
  \equiv y \pmod{\Gamma_i} \right\}$,
for $x,y \in \Gamma$.  This  yields,  for a subset
$U \subseteq \Gamma$, the \emph{Hausdorff
  dimension} $\hdim_{\Gamma}^\mathcal{S}(U) \in [0,1]$ with respect to the filtration series~$\mathcal{S}$.

Over the past twenty years, there have been interesting applications of Hausdorff dimension to profinite groups, starting with the pioneering work of Abercrombie~\cite{Abercrombie}, Barnea and Shalev~\cite{BaSh97}; see~\cite{KTZR} for a good overview. 
Barnea and Shalev~\cite{BaSh97} established the following
algebraic formula of the Hausdorff dimension of a
closed subgroup $H$ of $\Gamma$ as a logarithmic density:
\begin{equation*}
  \hdim_{\Gamma}^\mathcal{S}(H) =
  \varliminf_{i\to \infty} \frac{\log \lvert H\Gamma_i : \Gamma_i
    \rvert}{\log \lvert \Gamma : \Gamma_i \rvert},
\end{equation*}
where 
$\varliminf_{i \to \infty} a_i$ is the lower limit of a
sequence $(a_i)_{i \in \mathbb{N}}$ in
$\mathbb{R}$.

The \emph{Hausdorff spectrum} of $\Gamma$ with respect to $\mathcal{S}$ is
\[
\hspec^\mathcal{S}(\Gamma) = \{ \hdim_{\Gamma}^\mathcal{S}(H) \mid H \le_\mathrm{c} \Gamma\}
\subseteq [0,1],
\]
where $H$ runs through all closed subgroups of~$\Gamma$.  
Shalev~\cite[\S 4.7]{Sh00} was the first to consider the
\emph{normal Hausdorff spectrum} of~$\Gamma$, with respect to
$\mathcal{S}$, that is,
\[
\hspec^{\mathcal{S}}_{\trianglelefteq}(\Gamma) = \{ \hdim^{\mathcal{S}}_{\Gamma}(H) \mid H
\trianglelefteq_\mathrm{c} \Gamma \},
\]
which reflects the spread of Hausdorff dimensions of closed normal
subgroups in~$\Gamma$.  Until very recently,
little was known about normal Hausdorff spectra of finitely generated pro-$p$
groups. Indeed, early examples of normal Hausdorff spectra were all finite (see~\cite[\S 4.7]{Sh00}), until Klopsch and the second author~\cite{KT} constructed the first example of a finitely generated pro-$p$ group, for every prime $p$, with infinite normal Hausdorff spectra with respect to the five standard filtration series: the $p$-power series~$\mathcal{P}$,  the iterated $p$-power series~$\mathcal{I}$, the lower $p$-series~$\mathcal{L}$, the Frattini series~$\mathcal{F}$, and the dimension subgroup series~$\mathcal{D}$; we refer the reader to Section~2 for the definitions.
The normal Hausdorff spectra of the groups in~\cite{KT} consist of an interval $[0,\xi]$, for $\xi\le \nicefrac{1}{3}$, with one or two isolated points. The question was raised in~\cite{KT} whether a finitely generated pro-$p$ group, for $p$ any prime,  could be constructed with full normal Hausdorff spectra $[0,1]$. This was answered, for every odd prime $p$, by the first author and Klopsch~\cite{HK}. Their constructed group, however, does not produce the desired result for the case $p=2$.

In this paper, we settle the aforementioned question for $p=2$, by considering a modification of the construction in~\cite{HK}. Our group $G$ is a $2$-generated extension of an elementary abelian pro-$2$ group by the pro-$2$ wreath product $W= C_2 \mathrel{\hat{\wr}} \mathbb{Z}_2=\varprojlim_{k \in \N} C_2 \mathrel{\wr} C_{2^k}$.  We follow the strategy in~\cite{HK} to prove the result.

\begin{theorem}
The pro-$2$ group $G$ satisfies
$$
\hspec_{\trianglelefteq}^{\mathcal{S}}(G)= \hspec^{\mathcal{S}}(G)  =[0,1],
$$
for $\mathcal{S}\in\{\mathcal{M},\mathcal{L},\mathcal{D},\mathcal{P},\mathcal{F}
\}$.
\end{theorem}
\noindent Here $\mathcal{M}$ denotes a natural filtration series that arises from the construction of~$G$; see Sections~3 and~4 for details. Note for a pro-2 group, the iterated $2$-power series coincides with the Frattini series.

It was asked in~\cite[Prob.~5]{BaSh97}, whether there is a finitely generated pro-$p$ group~$\Gamma$ with  $\hspec^{\mathcal{P}}(\Gamma)=[0,1]$.
This was answered in the affirmative independently by Levai (see~\cite[\S 4.2]{Sh00}) and Klopsch (see~\cite{Klopsch}).
In the later case, it was shown that $C_p \mathrel{\hat{\wr}} \mathbb{Z}_p$ has full Hausdorff spectra with respect to the  $p$-power series~$\mathcal{P}$ (which equals the iterated $p$-power series~$\mathcal{I}$),  the Frattini series~$\mathcal{F}$, and the dimension subgroup series~$\mathcal{D}$, but that it does not have full Hausdorff spectrum with respect to the lower $p$-series~$\mathcal{L}$. The work of Klopsch and the second author~\cite{KT} give further examples of finitely generated pro-$p$ groups with full Hausdorff spectra with respect to~$\mathcal{P}$, $\mathcal{I}$, $\mathcal{F}$ and~$\mathcal{D}$, and the work of Garaialde Oca\~{n}a, Garrido and Klopsch~\cite{GGK} provide many examples of finitely generated pro-$p$ groups with full Hausdorff spectra with respect to~$\mathcal{F}$ and~$\mathcal{D}$.

As was the case for the groups in~\cite{HK}, our group here yields the first example of a finitely generated pro-$2$ group with full Hausdorff spectrum with respect to~$\mathcal{L}$.

\medskip

\noindent \emph{Organisation}.  Section~$2$ contains
preliminary results.  In Section~$3$ we give a presentation of the pro-$2$ group $G$ and describe a series
of finite quotients $G_k$, $k \in \mathbb{N}$, such that
$G = \varprojlim G_k$.
Lastly, in Section~$4$ we compute the normal Hausdorff
spectra of $G$ with respect to~$\mathcal{M}$, $\mathcal{L}$, $\mathcal{D}$,  $\mathcal{P}$ and $\mathcal{F}$. 

\medskip

\noindent \textit{Notation.}     All subgroups of profinite
groups are generally taken to be closed subgroups.  We use the notation $\le_\text{o}$ and $\le_\text{c}$ to denote open and closed subgroups respectively. 
Throughout, we  use
left-normed commutators, for example, $[x,y,z] = [[x,y],z]$.

\medskip

\noindent\textbf{Acknowledgements.} We thank the Heinrich-Heine-Universit\"{a}t  D\"{u}sseldorf, where a large part of this research was carried out. We also thank the referee for their helpful comments.

%%%%

\section{Preliminaries}

 Let $p$ be a prime. For $\Gamma$  a finitely generated pro-$p$ group, we
define below the four natural filtration series on $\Gamma$.  The \emph{$p$-power
  series} of $\Gamma$ is given by
\[
\mathcal{P} \colon \Gamma^{p^i} = \langle x^{p^i} \mid x \in \Gamma
\rangle, \quad i \in \mathbb{N}_0,
\]
where $\N_0=\N\cup\{0\}$.
The \emph{lower $p$-series} (or lower $p$-central
series) of $\Gamma$ is given recursively by
\begin{align*}
  \mathcal{L} \colon P_1(\Gamma) = \Gamma,  %
  & \quad \text{and} \quad  P_i(\Gamma) = P_{i-1}(\Gamma)^p
    \, [P_{i-1}(\Gamma),\Gamma] \quad
    \text{for $i \in\N_{\geq 2}$,} \\
  \intertext{and the \emph{Frattini series} of $\Gamma$ is
  given recursively by} 
  \mathcal{F} \colon \Phi_0(\Gamma) = \Gamma, %
  & \quad \text{and} \quad \Phi_i(\Gamma) = \Phi_{i-1}(\Gamma)^p
    \, [\Phi_{i-1}(\Gamma),\Phi_{i-1}(\Gamma)] \quad \text{for $i \in\N$.} \\
  \intertext{The (modular) \emph{dimension subgroup series} (or
  Jennings series or Zassenhaus series) of $\Gamma$ is defined recursively
  by}
  \mathcal{D} \colon D_1(\Gamma) = \Gamma, %
  & \quad \text{and} \quad D_i(\Gamma) = D_{\lceil i/p \rceil}(\Gamma)^p 
    \prod_{1 \le j <i} [D_j(\Gamma),D_{i-j}(\Gamma)] \quad \text{for $i \in\N_{\geq 2}$.}  
\end{align*}
In addition we set $P_0(\Gamma) = D_0(\Gamma) = \Gamma$.

Often an additional natural filtration series is considered: the \emph{iterated $p$-power series} of~$\Gamma$, which is defined by
\[
\mathcal{I} : I_0(\Gamma)=\Gamma,\quad\text{and}\quad I_j(\Gamma)=I_{j-1}(\Gamma)^p,\quad \text{for }j\in\N.
\]
Recall
for a pro-2 group~$\Gamma$, the iterated $2$-power series coincides with the Frattini series.

\medskip

Next, for convenience, we recall the following standard commutator identities. 

\begin{lemma}
\label{lemma commutator identities}
Let $\Gamma=\langle a, b\rangle$ be a group, let $p$ be any prime, and let $r\in\N$.
For $u,v\in \Gamma$, let $K(u,v)$ denote the normal closure in $\Gamma$ of (i) all commutators in $\{u,v\}$ of weight at least $p^r$ that have weight at least $2$ in $v$, together with (ii) the $p^{r-s+1}$th
powers of all commutators in $\{u,v\}$ of weight less than $p^s$ and of weight at least $2$ in $v$ for $1\le s\le r$.
Then
\begin{align}
 (ab)^{p^r}&\equiv_{K(a,b)}a^{p^r}b^{p^r}[b,a]^{\binom{p^r}{2}}[b,a,a]^{\binom{p^r}{3}}\cdots[b,a,
 \overset{p^r-2}{\ldots},a]^{\binom{p^r}{p^r-1}}[b,a,\overset{p^r-1}{\ldots},a], \label{equ:commutator-formula-1}\\
 [a^{p^r},b]&\equiv_{K(a,[a,b])}[a,b]^{p^r}[a,b,a]^{\binom{p^r}{2}}\cdots[a,b,a,
 \overset{p^r-2}{\ldots},a]^{\binom{p^r}{p^r-1}}[a,b,a,\overset{p^r-1}{\ldots},a].\label{equ:commutator-formula-2}
 \end{align}
\end{lemma}

\medskip

We  also recall the following definition from~\cite{KT}: for a countably
based infinite \mbox{pro-$p$} group $\Gamma$, equipped with a filtration series
$\mathcal{S} \colon \Gamma = \Gamma_0 \ge \Gamma_1\ge \cdots$, and a
closed subgroup $H \le_\mathrm{c} \Gamma$, we say that $H$ has \emph{strong
  Hausdorff dimension in $\Gamma$ with respect to $\mathcal{S}$} if
\[
\hdim^{\mathcal{S}}_{\Gamma}(H) = \lim_{i \to \infty} \frac{\log_p \lvert H
  \Gamma_i : \Gamma_i \rvert}{\log_p \lvert \Gamma : \Gamma_i \rvert}
\]
is given by a proper limit.

Lastly we note here a result that will be useful for the computation of normal Hausdorff spectra in the sequel.

\begin{proposition}\cite[Prop.~2.4]{HK}
\label{proposition path/area}
 Let $\Gamma$ be a countably based
  pro-$p$ group with an infinite abelian normal subgroup
  $Z \trianglelefteq_\mathrm{c} \Gamma$ such that $\Gamma/C_{\Gamma}(Z)$ is procyclic.
  Let $\mathcal{S} \colon \Gamma = S_0\ge S_1\ge \cdots$ be a filtration
  series of $\Gamma$ and consider the induced filtration series
  $\mathcal{S} |_Z \colon Z = S_0 \cap Z \ge S_1 \cap Z \ge\cdots$ of $Z$; for $i\in\N_0$, let $p^{e_i}$ be the exponent of $Z/(S_i \cap Z)$.  Suppose
  that, for every $i\in \N_0$, there exist $n_i \in \N$ and
  $K_i \le_\mathrm{c} \Gamma$ such that
  \[
  \gamma_{n_i + 1}(\Gamma) \cap Z \le S_i \cap Z \le K_i \qquad
  \text{and} \qquad
  \varliminf_{i\rightarrow\infty}\frac{e_in_i}{\log_p \lvert Z : K_i
    \cap Z \rvert} = 0.
  \]
  If $Z$ has strong Hausdorff dimension
  $\xi = \hdim_G^{\mathcal{S}}(Z) \in [0,1]$, then
  \[ 
  [0,\xi]\subseteq \hspec_{\trianglelefteq}^{\mathcal{S}}(\Gamma).
  \]
\end{proposition}

\smallskip

%%%%%

\section{The pro-\texorpdfstring{$2$}{2} group \texorpdfstring{$G$}{G}}\label{sec:3}

For $k\in\N$, let $\langle \dot x_k\rangle\cong C_{2^k}$ and $\langle \dot y_k\rangle\cong C_2$.
Define
$$W_k=\langle \dot y_k\rangle\wr\langle \dot x_k\rangle\cong B_k\rtimes\langle \dot x_k\rangle$$
where $B_k = \prod_{i = 0}^{2^k-1} \langle {\dot y_k}^{\,\dot  x_k^{\, i}}\rangle\cong C_2^{\, 2^k}$.
The structural results for the finite wreath products $W_k$ transfer naturally to the inverse limit $W\cong\varprojlim_k W_k$, 
i.e. the pro-$2$ wreath product
$$W=\langle \dot x,\dot y\rangle=B\rtimes \langle\dot x\rangle\cong C_2\ \hat{\wr}\ \Z_2$$
with top group $\langle\dot  x\rangle\cong\Z_2$ and base group $B=\langle \dot y^{\dot x^i}\mid i\in\Z\rangle\cong C_2^{\,\aleph_0}$; see~\cite[\S2.4]{KT} for further results.

Let $F_2=\langle a,b\rangle$ be the free pro-$2$ group on two generators and let $k\in\mathbb{N}$.
There exists a closed normal subgroup $R\trianglelefteq F_2$, respectively $R_k\trianglelefteq F_2$, such that 
\[
F_2/R\cong W, \quad\text{ respectively }\,\,
F_2/R_k\cong W_k,
\]
with $a$ corresponding to~$\dot x$, respectively $\dot x_k$, and $b$ corresponding to~$\dot y$, respectively $\dot y_k$.
 
Let $Y\ge R$ be the closed normal subgroup of $F_2$ such that $Y/R$ is
the pre-image of $B$  in $F_2/R$, and let  $Y_k\ge R_k$ be the closed normal subgroup of $F_2$ such that  $Y_k/R_k$ is
the pre-image of $B_k$ in $F_2/R_k$.

Consider now 
\[
N=[R,Y]R^2,
\]
 respectively 
 \[
 N_k=[R_k,Y_k]R_k^{\,2}\langle a^{2^k}\rangle^{F_2},
 \]
  and define 
  \[
  G=F_2/N, \quad\text{respectively }\,\,G_k=F_2/N_k.
  \]
  
We denote by $H$ and $Z$  the closed normal subgroups of $G$ corresponding 
to $Y/N$ and $R/N$, and we denote by $H_k$ and $Z_k$ the closed normal subgroups of  $G_k$ corresponding 
to  $Y_k/N_k$ and $R_k/N_k$. We denote the images of $a, b$ in~$G$, respectively in~$G_k$, by $x,y$, respectively $x_k,y_k$, so that  $G=\langle x,y\rangle$ and $G_k=\langle x_k,y_k\rangle$.

Note that the groups $G_k$ are finite for all $k\in\mathbb{N}$ and that they naturally form an inverse system so that $\varprojlim_k G_k=G$.
Furthermore we have $[H,Z]=1$, respectively $[H_k,Z_k]=1$.

\subsection{Properties of \texorpdfstring{$G_k$}{Gk}}
\label{subsection properties of G_k}

\begin{proposition}
\label{proposition order of Gk}
For $k\in\N$, the logarithmic order of $G_k$ is
$$
\log_2|G_k|=k+2^{k+1}+\binom{2^k}{2}= k+2^{2k-1}+2^{k+1}-2^{k-1}.
$$
\end{proposition}
\begin{proof}
We proceed as in~\cite[Lem.~5.1]{KT}. Since $G_k/Z_k\cong W_k\cong C_2\wr C_{2^k}$, we have
$$
\log_2|G_k|=\log_2|G_k/Z_k|+\log_2|Z_k|=k+2^k+\log_2|Z_k|.
$$
By construction, the subgroup $Z_k$ is elementary abelian, so we obtain
$$
Z_k=\langle \{  (y_k^{x_k^{\,i}})^2 \mid 0\le i \le 2^k-1\} \cup \{[y_k^{x_k^{\,i}},y_k^{x_k^{\,j}}]\mid 0\le i<j\le 2^{k}-1\}\rangle.
$$
Therefore $\log_2|G_k|\le k+2^{k+1}+\binom{2^k}{2}$.

For the converse, we will construct a factor group $\widetilde{G}_k$ of $G_k$ whose logarithmic order is $|\widetilde{G}_k|=k+2^{k+1}+\binom{2^k}{2}$.

We consider the finite $2$-group
$$
M=\langle \widetilde{y}_0,\ldots,\widetilde{y}_{2^k-1}\rangle=E/[\Phi(E),E]\Phi(E)^2,
$$
where $E$ is the free group on $2^k$ generators.
The images of $\widetilde{y}_0,\ldots,\widetilde{y}_{2^k-1}$ generate independently the elementary abelian quotient 
$M/\Phi(M)$, and the elements $\widetilde{y}_0^{\,2},\ldots,\widetilde{y}_{2^k-1}^{\,2}$ together with the 
commutators $[\widetilde{y}_i,\widetilde{y}_j]$ for $0\le i<j\le 2^k-1$ generate independently the elementary abelian 
subgroup $\Phi(M)$. 
The
  latter can be verified by considering homomorphisms from $M$ onto
  groups of the form $C_2^{\, 2^k-1} \times C_{4}$ and
  $C_2^{\, 2^k -2} \times \mathrm{Heis}(\mathbb{F}_2)$, where
  $\mathrm{Heis}(\mathbb{F}_2)$ denotes the group of upper
  unitriangular $3 \times 3$ matrices over $\mathbb{F}_2$. 
Next consider the faithful action of the cyclic group $X\cong\langle \widetilde{x}\rangle\cong C_{2^k}$ induced by
$$
\widetilde{y}_i^{\,\widetilde{x}}=
\begin{cases}
\widetilde{y}_{i+1} & \text{if }0\le i\le 2^k-2,\\
\widetilde{y}_0     & \text{if }i=2^{k}-1.
\end{cases}
$$
We define $\widetilde{G}_k=X\ltimes M$ and observe that $\log_2|\widetilde{G}_k|=k+2^{k+1}+\binom{2^k}{2}$.
Furthermore, it is easy to see that $\widetilde{G}_k/\Phi(M)\cong W_k$.
Thus, there is an epimorphism $\epsilon:G_k\rightarrow\widetilde{G}_k$ with $\epsilon(x_k)=\widetilde{x}$ and $\epsilon(y_k)=\widetilde{y}_0$, and since $|G_k|\le|\widetilde{G}_k|$ we conclude that $G_k\cong\widetilde{G}_k.$
\end{proof}

\begin{remark}
\label{remark derived subgroup}
 The proof of Proposition~\ref{proposition order of Gk} shows that $H_k^{\,2}=Z_k$ and, consequently, that $H^2=Z$. In particular, the exponent of $H_k$, and of $H$, is $4$. 
\end{remark}

Our next aim is to compute the lower central series of the groups $G_k$, and therefore of $G$. For that purpose, we recall the convenient notation introduced in~\cite{HK}, 
which will be used frequently in the paper.
We will denote $c_1=y$ and $c_i=[y,x,\overset{i-1}{\ldots},x]$ for $i\in\mathbb{N}_{\ge 2}$.
Similarly, we set
$c_{i,j}=[c_i,y,x,\overset{j-1}{\ldots},x]$ for every $j\in\mathbb{N}$.
For the sake of simplicity, for $k\in\N$, we will use the same notation for the corresponding elements in the group $G_k$. Thus, we will also write $c_1=y_k$, $c_i=[y_k,x_k,\overset{i-1}{\ldots},x_k]$,  
and $c_{i,j}=[c_i,y_k,x_k,\overset{j-1}{\ldots},x_k]$ for every $i\in\mathbb{N}_{\ge 2}$ and $j\in\mathbb{N}$, when working in the groups $G_k$, for $k\in\mathbb{N}$. It will be clear from the context whether our considerations apply to $G$ or $G_k$.

\begin{lemma}\label{exp-2}
 In the group $G_k$, for $k\in\mathbb{N}$,
 \begin{enumerate}
 \item[(i)] for $i\ge 2^{k-1}+1$,  we have 
 $c_i^{\,2}\in \gamma_{i+1}(G_k)$;
 \item[(ii)] for $i\ge 2^{k}+1$, we have $c_i^{\,2}=1$;
 \item[(iii)] for  $i\ge 2^k+2^{k-1}+1$, we have $c_i\in\gamma_{i+1}(G_k)$.
 \end{enumerate}
\end{lemma}
\begin{proof} 
(i) Observe that $\binom{2^k}{2^{k-1}}\equiv_4 2$ and that $\binom{2^k}{j}\equiv_4 0$ for any $1\le j\le 2^{k}-1$ with $j\neq 2^{k-1}$. As $H_k$ has exponent $4$ and $[H_k,H_k] \le Z_k$ exponent $2$, it follows from~(\ref{equ:commutator-formula-2}) that
  \begin{equation}\label{eq:sq-comm}
  1=[y_k,x_k^{\,2^k}]\equiv [y_k,x_k,\overset{2^{k-1}}\ldots,x_k]^2[y_k,x_k,\overset{2^k}\ldots,x_k] \pmod{\gamma_{2^{k}+2}(G_k)}.
 \end{equation}
Therefore $c_{2^{k-1}+1}^{\,2}\equiv c_{2^k+1}\pmod{\gamma_{2^k+2}(G_k)}$.

For $i\ge 2^{k-1}+2$, notice that
\begin{equation*}
    \begin{split}
    c_{i}^{\,2}= [y_k,x_k,\overset{i-1}\ldots, x_k]^2&\equiv [c_{2^{k-1}+1}^{\,2},x_k,\overset{i-2^{k-1}-1}\ldots,x_k] \pmod{\gamma_{i+1}(G_k)}\\
    &\equiv [c_{2^{k}+1},x_k,\overset{i-2^{k-1}-1}\ldots,x_k] \pmod{\gamma_{i+1}(G_k)}. 
    \end{split}
\end{equation*}
Hence the result follows.

(ii) This follows immediately from the fact that $c_i\in Z_k$ for $i\ge 2^k+1$; compare~\cite[Prop.~2.6(1)]{KT}.

(iii) It suffices to prove the result for  $i=2^k+2^{k-1}+1$. 
From~(\ref{eq:sq-comm}), we obtain
\begin{align*}
c_i= [c_{2^k+1},x_k,\overset{2^{k-1}}{\ldots},x_k]
\equiv [c_{2^{k-1}+1}^{\,2},x_k,\overset{2^{k-1}}{\ldots},x_k]\pmod{\gamma_{i+1}(G_k)}.
\end{align*}

As 
$$
[c_{2^{k-1}+1}^{\,2},x_k,\overset{2^{k-1}}{\ldots},x_k]
\equiv
c_{2^k+1}^{\,2}\pmod{\gamma_{i+1}(G_k)}
$$
we have 
$c_i\equiv c_{2^{k}+1}^{\,2}\pmod{\gamma_{i+1}(G_k)},$
and by (ii), it follows that  $c_i\in\gamma_{i+1}(G_k)$, as required.
\end{proof}

Let us write $z_{i,j}=[c_i,c_j]$ for every $i,j\in\N$.
From~\cite[Prop.~2.6(1)]{KT} we have $H=\langle c_n\mid n\in \mathbb{N}\rangle$, and from Remark~\ref{remark derived subgroup} we deduce that
$$
Z=\langle c_n^{\,2},z_{i,j}\mid  n,i,j\in\mathbb{N}   \text{ with } j<i \rangle.
$$

We recall the following result from~\cite[Lem.~4.3]{HK}, which although was written in the setting of the pro-$p$ group constructed in~\cite{HK}, for $p$ an odd prime, the same proof holds for our group~$G$.
 
\begin{lemma}
\label{lemma double product}
In the group~$G$, for $i,j\in \mathbb{N}$ we have
$$
[z_{i,j},x,\overset{k}{\ldots},x]=\prod_{s=0}^{k}\left(\prod_{n=0}^{s}z_{i+k-n,j+k-s+n}^{ \binom{k}{s}\binom{s}{n}}\right)\quad \text{for every }k\in \mathbb{N}_0.
$$
\end{lemma}

\begin{corollary}
\label{corollary p^k}
In the group~$G$, for $i,j\in \mathbb{N}$ we have
$$
[z_{i,j},x^{2^k}]=z_{i+{2^k},j}z_{i,j+2^k}z_{i+2^k,j+2^k}\quad \text{for every }k\in\mathbb{N}_0.
$$
\end{corollary}

\begin{proof}
As was reasoned in~\cite{HK}, this result follows directly from~(\ref{equ:commutator-formula-2}) and Lemma~\ref{lemma double product}.
\end{proof}

\begin{lemma}
\label{lemma c_{n,2^k}}
Let $k\in\mathbb{N}$. In the group~$G_k$, for   $m\in \mathbb{N}$ even, we have
$$
c_{m,2^k}\in\gamma_{2^k+m+1}(G_k).
$$
\end{lemma}

\begin{proof}
Note that
$$
c_{m,2^k}=[z_{m,1},x_k,\overset{2^{k}-1}{\ldots},x_k],
$$
and since  $z_{i,j}\in\gamma_{i+j}(G_k)$ for every $i,j\in \mathbb{N}$, we have by Lemma~\ref{lemma double product} that
$$
[z_{m,1},x_k,\overset{2^k-1}{\ldots},x_k]\equiv\prod_{n=0}^{2^{k}-1}z_{m+2^k-1-n,1+n}^{\binom{2^k-1}{n}}\pmod{\gamma_{2^k+m+1}(G_k)}.
$$
In addition, the exponent of~$Z_k$ is~$2$ by construction, so since all the binomial numbers~$\binom{2^k-1}{n}$ are odd, we get
$$
[z_{m,1},x_k,\overset{2^k-1}{\ldots},x_k]\equiv\prod_{n=0}^{2^{k}-1}z_{m+2^k-1-n,1+n}\pmod{\gamma_{2^k+m+1}(G_k)}.
$$
Recall that $c_i\in Z_k$ for every $i\ge 2^k+1$ by~\cite[Prop.~2.6(1)]{KT}, so $z_{m+2^k-1-n,1+n}=1$ for all $n\le m-2$. Thus,
\begin{align*}
[z_{m,1},x_k,\overset{2^k-1}{\ldots},x_k]&\equiv\prod_{n=m-1}^{2^{k}-1}z_{m+2^k-1-n,1+n}\pmod{\gamma_{2^k+m+1}(G_k)}\\
&\equiv\prod_{n=0}^{2^{k}-m}z_{2^k-n,m+n}\pmod{\gamma_{2^k+m+1}(G_k)}.
\end{align*}
As $z_{i,j}=z_{j,i}$ for all $i,j\in \mathbb{N}$ and since $m$ is even, we finally obtain
$$
[z_{m,1},x_k,\overset{2^k-1}{\ldots},x_k]\equiv 1\pmod{\gamma_{2^k+m+1}(G_k)},
$$
as required.
\end{proof}

\begin{proposition}
\label{proposition G_k lower central series}
For $k\in\N$,
the nilpotency class of $G_k$ is $2^{k+1}-1$ and the lower central series of $G_k$ satisfies:
\begin{itemize}
    \item $\gamma_1(G_k)=G_k=\langle x_k,y_k\rangle\gamma_2(G_k)$ with
    $$
    \gamma_1(G_k)/\gamma_2(G_k)\cong C_{2^k}\times C_4.
    $$
    \item If $2\le i\le 2^{k}$, then
    $$
    \gamma_i(G_k)=
    \begin{cases}
    \langle c_i,c_{2,i-2},c_{4,i-4},\ldots, c_{i-2,2}\rangle\gamma_{i+1}(G_k) & \text{if }i\equiv_2 0,\\
    \langle c_{i},c_{2,i-2},c_{4,i-4},\ldots, c_{i-1,1}\rangle\gamma_{i+1}(G_k)& \text{if }i\equiv_2 1,\\
    \end{cases}
    $$
    with
    \begin{equation}
    \label{eq 1}
    \gamma_i(G_k)/\gamma_{i+1}(G_k)\cong
    \begin{cases}
    C_4\times C_2\times\overset{(i-2)/2}{\dots}\times C_2 &\text{if } 2\le i\le 2^{k-1}\text{ and }i\equiv_2 0,\\
    C_4\times C_2\times\overset{(i-1)/2}{\dots}\times C_2 &\text{if } 2\le i\le 2^{k-1}\text{ and }i\equiv_2 1,\\
     C_2\times\overset{i/2}{\dots}\times C_2& \text{if }2^{k-1}+1\le i\le 2^{k}\text{ and }i\equiv_2 0,\\
     C_2\times\overset{(i+1)/2}{\dots}\times C_2& \text{if }2^{k-1}+1\le i\le 2^{k}\text{ and }i\equiv_2 1.
    \end{cases}
    \end{equation}

    \item If $2^k+1\le i\le 2^k+2^{k-1}$, then
    \begin{align*}
    &\gamma_i(G_k)=\\
    &
    \begin{cases}
    \langle c_i, c_{i-2^k+2,2^{k}-2},c_{i-2^k+4,2^{k}-4},\ldots,c_{2^{k}-2,i-2^{k}+2} , c_{2^k,i-2^k}\rangle\gamma_{i+1}(G_k) &\text{if }i\equiv_2 0,\\
    \langle c_i,  c_{i-2^k+1,2^{k}-1},c_{i-2^k+3,2^{k}-3},\ldots, c_{2^{k}-2,i-2^{k}+2} , c_{2^{k},i-2^{k}} \rangle\gamma_{i+1}(G_k)&\text{if }i\equiv_2 1,
    \end{cases}
    \end{align*}
    with
    \begin{equation}
    \label{eq 2}
    \gamma_{i}(G_k)/\gamma_{i+1}(G_k)\cong \begin{cases}
    C_2\times\overset{(2^{k+1}-i+2)/2}{\dots}\times C_2 &\text{if }i\equiv_2 0,\\
    C_2\times\overset{(2^{k+1}-i+3)/2}{\dots}\times C_2 &\text{if }i\equiv_2 1.
    \end{cases}
    \end{equation}
    
    \item If $2^k+2^{k-1}+1\le i\le 2^{k+1}$, then
    \begin{align*}
    &\gamma_i(G_k)=\\
    &\begin{cases}\langle c_{i-2^k+2,2^k-2},c_{i-2^k+4,2^k-4},\ldots, c_{2^{k}-2,i-2^{k}+2}, c_{2^k,i-2^k}\rangle\gamma_{i+1}(G_k) &\text{if }i\equiv_2 0,\\
    \langle c_{i-2^k+1,2^k-1},c_{i-2^k+3,2^k-3},\ldots,c_{2^{k}-2,i-2^{k}+2}, c_{2^k,i-2^k}\rangle\gamma_{i+1}(G_k) & \text{if }i\equiv_2 1,
    \end{cases}
    \end{align*}
    with
    \begin{equation}
    \label{eq 3}
    \gamma_{i}(G_k)/\gamma_{i+1}(G_k)\cong \begin{cases}
    C_2\times\overset{(2^{k+1}-i)/2}{\dots}\times C_2 &\text{if }i\equiv_2 0,\\
    C_2\times\overset{(2^{k+1}-i+1)/2}{\dots}\times C_2 &\text{if }i\equiv_2 1.
    \end{cases}
    \end{equation}
\end{itemize}
\end{proposition}
\begin{proof}
The first assertion is obvious.
To prove the other ones, we start by showing that
$$\gamma_i(G_k)=
\begin{cases}
\langle c_i,c_{2,i-2},c_{4,i-4},\ldots, c_{i-2,2}\rangle\gamma_{i+1}(G_k) &\text{if }i\equiv_2 0,\\
\langle c_{i},c_{2,i-2},c_{4,i-4},\ldots, c_{i-1,1}\rangle\gamma_{i+1}(G_k) &\text{if }i\equiv_2 1.
\end{cases}
$$
We proceed by induction on $i$.
For $i=2$, we have $\gamma_2(G_k)=\langle c_2\rangle\gamma_3(G_k)$ and for $i=3$ we have $\gamma_3(G_k)=\langle c_3,c_{2,1}\rangle\gamma_4(G_k)$.
Assume then that $i>3$, that $i$ is even, and that
$$\gamma_{i-2}(G_k)=\langle c_{i-2},c_{2,i-4},c_{4,i-6},\ldots, c_{i-4,2}\rangle\gamma_{i-1}(G_k)$$
and
$$\gamma_{i-1}(G_k)=\langle c_{i-1},c_{2,i-3},c_{4,i-5},\ldots, c_{i-2,1}\rangle\gamma_{i}(G_k).$$
Note that $c_{n,m}\in [H_k,H_k]\le Z_k$ for all $n,m\in \mathbb{N}$, so $[c_{n,m},y_k]=1$.
Thus, 
$$\gamma_i(G_k)=\langle c_i,c_{i-1,1},c_{2,i-2},c_{4,i-4},\ldots,c_{i-2,2}\rangle\gamma_{i+1}(G_k).$$
For convenience, we write
$$
M=\langle c_i,c_{2,i-2},c_{4,i-4},\ldots, c_{i-2,2}\rangle\gamma_{i+1}(G_k),
$$
and we have to check that $c_{i-1,1}\in M$.
Notice that $c_{i-1,1}=[c_{i-2},x_k,y_k]$, and the Hall-Witt identity yields
$$[c_{i-2},x_k,y_k][x_k,y_k,c_{i-2}][y_k,c_{i-2},x_k]\equiv 1\pmod{M}.$$
Note also that $[y_k,c_{i-2},x_k]\equiv c_{i-2,2}^{-1}\equiv 1\pmod{M}$, so 
$$c_{i-1,1}\equiv [c_{i-2},c_2]^{-1}\pmod{M}.$$
Let us prove that $[c_m,c_n]\equiv 1\pmod{M}$ for all $n\ge 2$ with $n\le m$ and $n+m=i$.
We argue by induction on $m-n$.
If $m-n=0$ then $n=m$ and $[c_m,c_n]=1$.
Now suppose that $m-n>0$, which, since $i$ is even, implies that $m-n\ge 2$.
Observe that $[c_m,c_n]=[c_{m-1},x_k,c_n]$, so again by the Hall-Witt identity,
$$[c_{m-1},x_k,c_n][x_k,c_n,c_{m-1}][c_n,c_{m-1},x_k]\equiv 1\pmod{M}.$$
Since $0\le(m-1)-(n+1)=m-n-2<m-n$, we have
$$
[x_k,c_n,c_{m-1}]\equiv[c_{m-1},c_{n+1}]\equiv 1\pmod{M}
$$
by induction.
Hence $[c_m,c_n]\equiv[c_n,c_{m-1},x_k]^{-1}\pmod{M}$.
As
$$
[c_n,c_{m-1}]\in\gamma_{i-1}(G_k)\cap Z_k,
$$
it follows that
$$
[c_n,c_{m-1}]\equiv c_{i-1}^{\,n_0}c_{2,i-3}^{\,n_2}c_{4,i-5}^{\,n_4}\cdots c_{i-2,1}^{\,n_{i-2}}\pmod{\gamma_i(G_k)}
$$
for some $n_0,n_2,\ldots,n_{i-2}\in\Z$. Hence
\[
 [c_m,c_n]\equiv[c_n,c_{m-1},x_k]\equiv c_{i}^{\,n_0}c_{2,i-2}^{\,n_2}c_{4,i-4}^{\,n_4}\cdots c_{i-2,2}^{\,n_{i-2}}\equiv 1\pmod{M},
\]
as we wanted to prove.
This, in particular, implies that $c_{i-1,1}\in M$, as claimed.
Hence, 
$$\gamma_i(G_k)=\langle c_i,c_{2,i-2},c_{4,i-4},\ldots, c_{i-2,2}\rangle\gamma_{i+1}(G_k)$$
and, again, since $[c_{n,m},y_k]=1$, we have
$$\gamma_{i+1}(G_k)=\langle c_{i+1},c_{2,i-1},c_{4,i-3},\ldots, c_{i-2,3},c_{i,1}\rangle\gamma_{i+2}(G_k).$$

Now, for $2^{k}+1\le i\le 2^{k+1}$, we add $x_k$ and $y_k$ to the commutators that inductively generate $\gamma_{i-1}(G_k)$ modulo~$\gamma_i(G_k)$. Removing the unnecessary generators according to Lemmata~\ref{exp-2} and~\ref{lemma c_{n,2^k}}, one can deduce that the generators of $\gamma_i(G_k)$ modulo~$\gamma_{i+1}(G_k)$ are precisely the stated ones. 

Finally, it suffices to check that the isomorphisms (\ref{eq 1}), (\ref{eq 2}) and (\ref{eq 3}) are satisfied.
The generators we have found for the terms of the lower central series, together with Lemma~\ref{exp-2}, show that any quotient of two consecutive terms of the lower central series is actually isomorphic to a quotient of the corresponding abelian group in the statement.
In particular, we obtain upper bounds for the logarithmic orders of the quotients of two consecutive terms of the lower central series.
In fact, all these upper bounds sum to the logarithmic order of $G_k$.
Indeed, we have
\begin{align*}
\sum_{i=1}^{2^{k+1}} \log_2&|\gamma_i(G_k):\gamma_{i+1}(G_k)|\\
&=(k+2) + \Big(2^{k-1}-1+\sum_{i=2}^{2^k}\lceil i/2\rceil \Big)+ \Big( 2^{k-1}+\sum_{i=2^k+1}^{2^{k+1}}\lceil(2^{k+1}-i)/2\rceil \Big) \\
&=k+2^{k+1}+\binom{2^k}{2},
\end{align*}
which is precisely, by Proposition~\ref{proposition order of Gk} the logarithmic order of $G_k$. Hence, isomorphisms (\ref{eq 1}), (\ref{eq 2}) and (\ref{eq 3})
are satisfied and the proof is finished.
\end{proof}

\begin{remark}
\label{remark index}
From Proposition~\ref{proposition G_k lower central series} we deduce that the logarithmic order of $Z/(\gamma_i(G)\cap Z)$ is
$$
2\Big(1+2+\cdots +\dfrac{i-1}{2}\Big)=2\binom{\nicefrac{(i+1)}{2}}{2}
$$
if $i$ is odd or
$$
2\Big(1+2+\cdots +\dfrac{i-2}{2}\Big)+\dfrac{i}{2}=2\binom{\nicefrac{i}{2}}{2}+\frac{i}{2}
$$
if $i$ is even.
\end{remark}

We include a result which highlights a further difference between the group $G$ and the group constructed in~\cite{HK}.

\begin{lemma} 
For $k\in\mathbb{N}$, the group $G_k$ has exponent $2^{k+2}$.
\end{lemma}

\begin{proof}
 First we show that $x_ky_k$ has order $2^{k+2}$. Consider $(x_ky_k)^{2^k}$ and observe that $K(x_k,y_k)\le [H_k,H_k]$
 in (\ref{equ:commutator-formula-1}) has exponent 2. Therefore, $(x_ky_k)^{2^k}\equiv_{K(x_k,y_k)}[y_k,x_k,\overset{2^{k-1}-1}\ldots,x_k]^2
 [y_k,x_k,\overset{2^k-1}\ldots,x_k]$ yields $(x_ky_k)^{2^{k+1}}=c_{2^k}^{\,2}$, which is non-trivial by the proof of
 Proposition~\ref{proposition G_k lower central series}. Hence the result.
 
 For a general element $g=x_k^{\,i\,}h$ for some $0\le i \le 2^k-1$ and $h\in H_k$, it similarly follows that $g^{2^{k+2}}=1$.
\end{proof}

The following two results will be needed for computing the normal Hausdorff spectra of~$G$ with respect to the series~$\mathcal{L}$ and~$\mathcal{D}$ in the next section.

\begin{proposition}
\label{proposition G_k p-central series}
For $k\in\N$, the length of the lower $2$-series of $G_k$ is $2^{k+1}-1$ and
\begin{align*}
 P_1(G_k)&=G_k,\\
 P_2(G_k)&=\langle x_k^{\,2},y_k^{\,2}\rangle \gamma_2(G_k),
\end{align*}
and
$$
P_i(G_k)=
\begin{cases}
\langle x_k^{\,2^{i-1}}, c_{i-1}^{\,2}\rangle\gamma_i(G_k) &\text{for }3\le i\le 2^{k-1}+1,\\
\langle x_k^{\,2^{i-1}}\rangle\gamma_i(G_k) &\text{for }2^{k-1}+2\le i\le 2^{k+1}.
\end{cases}
$$
\end{proposition}

\begin{proof}
If $i=1$ or $2$, the results are obvious, so consider $i=3$. As 
$$
[\langle x_k^{\,2},y_k^{\,2}\rangle,G_k]\le\gamma_2(G_k)^2\gamma_3(G_k),
$$
it suffices to show that $\langle x_k^{\,2},y_k^{\,2}\rangle^2\le 
\langle x_k^{\,4}\rangle\gamma_3(G_k)$ and  $\gamma_2(G_k)^2\le \langle c_2^{\,2}\rangle\gamma_3(G_k)$.
Note that
\[
[y_k^2,x_k^2]\equiv [y_k,x_k]^4=1\pmod {\gamma_3(G_k)},
\]
and since $x_k^4,y_k^4\in\langle x_k^4\rangle\gamma_3(G_k)$, the first inclusion holds.
For the second statement, it follows from Proposition~\ref{proposition G_k lower central series}, as $[y_k,x_k]$ is the only generator of $\gamma_2(G_k)$ modulo~$\gamma_3(G_k)$.

Now let $4\le i\le 2^{k-1}+1$ and assume by induction that
$$
P_{i-1}(G_k)=\langle x_k^{\,2^{i-2}}, c_{i-2}^{\,2}\rangle\gamma_{i-1}(G_k).
$$
On the one hand,
$$
[P_{i-1}(G_k),G_k]=[\langle x_k^{\,2^{i-2}}, c_{i-2}^{\,2}\rangle\gamma_{i-1}(G_k),G_k]=[\langle x_k^{\,2^{i-2}}, c_{i-2}^{\,2}\rangle,G_k]\gamma_{i}(G_k)
$$
and Proposition~\ref{proposition G_k lower central series} and (\ref{equ:commutator-formula-2}) yield 
\[
[\langle x_k^{\,2^{i-2}}, c_{i-2}^{\,2}\rangle,G_k]\gamma_{i}(G_k)=[\langle  c_{i-2}^{\,2}\rangle,G_k]\gamma_{i}(G_k).
\]
Then by similar arguments as above, one deduces that 
\[
[c_{i-2}^{\,2},G_k]\gamma_i(G_k)= \langle 
c_{i-1}^{\,2}\rangle \gamma_i(G_k).
\]

On the other hand,
$$
P_{i-1}(G_k)^2\equiv \langle x_k^{\,2^{i-2}}\rangle^2\gamma_{i-1}(G_k)^2\equiv \langle x_k^{\,2^{i-1}},  c_{i-1}^{\,2}\rangle\pmod{\gamma_i(G_k)},
$$
so we conclude that
$$
P_i(G_k)=\langle x_k^{\,2^{i-1}}, c_{i-1}^{\,2}\rangle\gamma_i(G_k),
$$
as asserted. The case $2^{k-1}+2\le i\le 2^{k+1}$ follows similarly, using Lemma~\ref{exp-2}(ii).
\end{proof}

\begin{proposition}
\label{proposition G_k dimension subgroup series}
For $k\in\N$, the length of the dimension subgroup series of $G_k$ is $2^{k+1}$ and
$$
D_i(G_k)=\langle x_k^{\,2^{l(i)}}\rangle\gamma_{\lceil \nicefrac{i}{2}\rceil}(G_k)^2\gamma_i(G_k) \quad \text{for } 1\le i\le 2^{k+1},
$$
where $l(i)=\lceil\log_2 i\rceil$.
\end{proposition}
\begin{proof}
By~\cite[Thm.~11.2]{DDMS}, we have
$$
D_i(G_k)=\prod_{n\cdot 2^m\ge i}\gamma_n(G_k)^{2^m}
$$
for every $i\in \mathbb{N}$, and since $\exp(\gamma_2(G_k))=4$, we obtain
$$
D_i(G_k)=G_k^{\,2^{l(i)}}\gamma_{\lceil \nicefrac{i}{2}\rceil}(G_k)^2\gamma_i(G_k).
$$
The result is clear for $i=1,2$, so we assume $i\ge 3$. By~(\ref{equ:commutator-formula-1}), for every $a,b\in G_k$ it follows that
$$
(ab)^{2^{l(i)}}=  a^{2^{l(i)}}b^{2^{l(i)}}[b,a,
 \overset{2^{l(i)-1}-1}{\ldots},a]^{\binom{2^{l(i)}}{2^{l(i)-1}}} c
$$
with $c\in\gamma_{2^{l(i)}}(G_k)$. Since $[b,a,
 \overset{2^{l(i)-1}-1}{\ldots},a]^{\binom{2^{l(i)}}{2^{l(i)-1}}} \in \gamma_{\lceil \nicefrac{i}{2}\rceil}(G_k)^2$ and $\gamma_{2^{l(i)}}(G_k)\le\gamma_i(G_k)$, we get
$$
D_i(G_k)=\langle x_k^{\,2^{l(i)}}\rangle\gamma_{\lceil \nicefrac{i}{2}\rceil}(G_k)^2\gamma_i(G_k),
$$
as required.
\end{proof}

%%%%

\section{The normal Hausdorff spectra of \texorpdfstring{$G$}{G}}\label{sec:normal-spectra}

In this section we compute the normal Hausdorff spectra of~$G$ with respect to the filtration series $\mathcal{M}$, $\mathcal{L}$, $\mathcal{D}$, $\mathcal{P}$, and $\mathcal{F}$. Here $\mathcal{M}:M_0\ge M_1\ge\cdots$ stands for the natural filtration series of~$G$ where each $M_i$ is the subgroup of~$G$ corresponding to~$N_i/N$, where here $N_0=F_2$.

 As an easy illustration of our methods we start computing the normal Hausdorff spectrum of~$G$ with respect to~$\mathcal{M}$.

\begin{theorem}\label{M}
The pro-$2$ group $G$ satisfies
$$
\hspec_{\trianglelefteq}^{\mathcal{M}}(G)=[0,1]
$$
and $Z$ has strong Hausdorff dimension~$1$ in~$G$ with respect to~$\mathcal{M}$.
\end{theorem}

\begin{proof}
By Proposition~\ref{proposition G_k lower central series} we know that $\gamma_{2^{k+1}}(G)\le M_k$ for all $k\in\N$.
On the other hand, by the proof of Proposition~\ref{proposition order of Gk} we have
$$
\log_2|Z:M_k\cap Z|=\log_2|ZM_k:M_k|=2^k+\binom{2^k}{2}.
$$
In particular
$$
\varliminf_{k\rightarrow\infty}\dfrac{2^{k+1}}{\log_2|Z:M_k\cap Z|}=0.
$$
Moreover,
\begin{align*}
\hdim_{G}^{\mathcal{M}}(Z)=\varliminf_{k\rightarrow\infty}\dfrac{\log_2|ZM_k:M_k|}{\log_2|G:M_k|}
=\varliminf_{k\rightarrow\infty}\dfrac{2^k+\binom{2^k}{2}}{k+2^{k+1}+\binom{2^k}{2}}
=1.
\end{align*}
Therefore, the second statement follows and, by Proposition~\ref{proposition path/area}, we conclude that
\[
\hspec_{\trianglelefteq}^{\mathcal{M}}(G)=[0,1]. \qedhere
\]
\end{proof}

\begin{theorem}\label{LD}
For $\mathcal{S}\in\{\mathcal{L},\mathcal{D}\}$, the pro-$2$ group $G$ satisfies
$$
\hspec_{\trianglelefteq}^{\mathcal{S}}(G)=[0,1]
$$
and $Z$ has strong Hausdorff dimension~$1$ in~$G$ with respect to~$\mathcal{S}$.
\end{theorem}

\begin{proof}
By Remark~\ref{remark index} we have 
$$
\varliminf_{k\rightarrow\infty}\dfrac{k}{\log_2|Z:P_k(G)\cap Z|}=0 \quad\text{and}\quad 
\varliminf_{k\rightarrow\infty}\dfrac{k}{\log_2|Z:D_k(G)\cap Z|}=0,
$$
and they are furthermore given by proper limits.
Then it suffices by Proposition~\ref{proposition path/area} to show that $Z$ has strong Hausdorff dimension~$1$ with respect to~$\mathcal{S}$.
Write  $G=S_0\ge S_1\ge S_2\ge\cdots$ for the subgroups of the filtration series $\mathcal{S}$. Observe that by~\cite[Prop.~2.6]{KT} we have
$$
\log_2|G:S_kZ|=2k,
$$
and so
\begin{align*}
    \lim_{k\rightarrow\infty}\dfrac{\log_2|G:S_kZ|}{\log_2|Z:S_k\cap Z|}
        =0,
\end{align*}
is given by a proper limit. 
Thus,
\begin{align*}
\hdim_{G}^{\mathcal{S}}(Z)&=\varliminf_{k\rightarrow\infty}\left(\dfrac{\log_2|G:S_k|}{\log_2|S_kZ:S_k|}\right)^{-1}\\
&=\varliminf_{k\rightarrow\infty}\left(\dfrac{\log_2|G:S_kZ|+\log_2|S_kZ:S_k|}{\log_2|S_kZ:S_k|}\right)^{-1}\\
&=\varliminf_{k\rightarrow\infty}\left(\dfrac{\log_2|G:S_kZ|}{\log_2|Z:S_k\cap Z|}+1\right)^{-1}
=1,
\end{align*}
and $Z$ has strong Hausdorff dimension 1, as we wanted.
Thus, the proof is complete.
\end{proof}

For all $n\in \mathbb{N}$ define $\Gamma_{n}=\langle x^{2^n}\rangle Q_{n-1}\gamma_{2^n}(G)\trianglelefteq G$ where
$$Q_{n-1}=\langle c_{i}^{\,2} \mid i\ge 2^{n-1}\rangle.
$$

\begin{lemma}
\label{lemma Gamma}
For each $n\in \mathbb{N}$, we have $\Gamma_n^{\,2}\le \Gamma_{n+1}$.
\end{lemma}
\begin{proof}
We only have to check that $\Gamma_n'\le \Gamma_{n+1}$. Clearly $[Q_{n-1},Q_{n-1}\gamma_{2^n}(G)]=1$ and $[\gamma_{2^n}(G),\gamma_{2^n}(G)]\le\gamma_{2^{n+1}}(G)\le\Gamma_{n+1}$, so it suffices to prove that
$$
[\langle x^{2^n}\rangle,Q_{n-1}\gamma_{2^n}(G)]\le\Gamma_{n+1}.
$$

On the one hand, (\ref{equ:commutator-formula-2}) yields
$$
[\gamma_{2^n}(G),x^{2^n}]\le\gamma_{2^{n}+2^{n-1}}(G)^2\gamma_{2^{n+1}}(G)\le\Gamma_{n+1}.
$$

On the other hand, for $2^{n-1}\le i\le 2^{n}-1$, again by~(\ref{equ:commutator-formula-2}) we have
$$
[c_i,x^{2^n}]\in\gamma_{2^n}(G)^2\gamma_{2^n+2^{n-1}}(G),
$$
so
$$
[c_i^{\,2},x^{2^n}]=[c_i,x^{2^n}]^2[c_i,x^{2^n},c_i]\in\Gamma_{n+1},
$$
as required.
\end{proof}

\begin{theorem}\label{P}
The pro-$2$ group $G$ satisfies
$$
\hspec_{\trianglelefteq}^{\mathcal{P}}(G)=[0,1]
$$
and $Z$ has strong Hausdorff dimension $1$ in $G$ with respect to $\mathcal{P}$.
\end{theorem}

\begin{proof}
An arbitrary element of $G$ can be written as $x^ih$ with $h\in H$ and $i\in \mathbb{Z}_2$, and by~(\ref{equ:commutator-formula-1}), it follows that $(x^ih)^{2^k}\in\Gamma_{k}$ for $k\in\mathbb{N}$.
Then $G^{2^k}\le\Gamma_{k}$ and, in particular, $G^{2^k}\cap Z\le \Gamma_{k}\cap Z$.
It is easy to see that
$$
\Gamma_{k}\cap Z=(\gamma_{2^{k-1}}(G)^2\gamma_{2^k}(G))\cap Z=\gamma_{2^{k-1}}(G)^2(\gamma_{2^k}(G)\cap Z),
$$
and since $[\gamma_{2^{k-1}}(G),\gamma_{2^{k-1}}(G)]\le\gamma_{2^k}(G)\cap Z$, it follows that
$$
\gamma_{2^{k-1}}(G)^2\le\langle c_{2^{k-1}}^{\,2},c_{2^{k-1}+1}^{\,2},\ldots,c_{2^k-1}^{\,2}\rangle(\gamma_{2^k}(G)\cap Z).
$$
Thus, by Remark \ref{remark index}, we have
$$
\log_2|Z:\Gamma_k\cap Z|
=\log_2|Z:\gamma_{2^{k-1}}(G)^2(\gamma_{2^k}(G)\cap Z)| =2\binom{2^{k-1}}{2}.
$$

On the other hand, from the construction of $G$ and $G_k$ it can be deduced easily that $G/G^{2^k}\cong G_k/G_k^{\,2^k}$. Indeed, $N_k=N\langle a^{2^k}\rangle^{F_2}$ and $\langle a^{2^k}\rangle^{F_2}\le F_2^{\,2^k}$.
Hence, by Proposition~\ref{proposition G_k lower central series}, we get $\gamma_{2^{k+1}}(G)\le G^{2^k}$.

Now,
$$
\lim_{k\rightarrow\infty}\dfrac{2^{k+1}}{\log_2|Z:\Gamma_{k}\cap Z|}=0,
$$
so again by Proposition~\ref{proposition path/area}
it suffices to check that $Z$ has strong Hausdorff dimension~$1$ with respect to~$\mathcal{P}$.
Note that
$$
\log_2|G:G^{2^k}Z|=\log_2|G_k:G_k^{\,2^k}Z_k|\le\log_2|W_k|= k+2^k,
$$
and so,
\begin{align*}
    \lim_{k\rightarrow\infty}\dfrac{\log_2|G:G^{2^k}Z|}{\log_2|Z:G^{2^k}\cap Z|}\le\lim_{k\rightarrow\infty}\dfrac{\log_2|G:G^{2^k}Z|}{\log_2|Z:\Gamma_{k}\cap Z|}=0
\end{align*}
is given by a proper limit.
Thus, the result follows as in the proof of Theorem~\ref{LD}.
\end{proof}

\begin{theorem}\label{F}
The pro-$2$ group $G$ satisfies
$$
\hspec_{\trianglelefteq}^{\mathcal{F}}(G)=[0,1]
$$
and $Z$ has strong Hausdorff dimension $1$ in $G$ with respect to $\mathcal{F}$.
\end{theorem}
\begin{proof}
We claim that
$$
T_k(\gamma_{2^{k}+2^{k-1}-1}(G)\cap Z)\le\Phi_k(G)\le \Gamma_k
$$
where
$$
T_k=\langle x^{2^k},c_i^{\,2},c_j\mid i\ge 2^{k-1},\, j\ge 2^k\rangle
$$
for all $k\in \mathbb{N}$.
We will proceed by induction on $k$.
If $k=1$ the result is clear, so assume $k\ge 2$.
On the one hand, it follows from Lemma~\ref{lemma Gamma} that
\begin{align*}
    \Phi_k(G)=\Phi_{k-1}(G)^2&\le\Gamma_{k-1}^{\,2}\le\Gamma_{k}.
\end{align*}
Hence, we only need to check that
$$
T_k(\gamma_{2^{k}+2^{k-1}-1}(G)\cap Z)\le \Delta,
$$
where
$$
\Delta=\Phi\big(T_{k-1}(\gamma_{2^{k-1}+2^{k-2}-1}(G)\cap Z)\big).
$$
Of course we have $x^{2^k},c_i^{\,2}\in\Delta$ for all $i\ge 2^{k-1}$.
We also have $T_{k-1}'\le\Delta$, so $\langle z_{i,j}\mid i>j\ge 2^{k-1}\rangle\le\Delta$.
Let us see that $z_{i,j}\in\Delta$ whenever $i>j$, $i+j\ge 2^{k}+2^{k-1}-1$ and $j\le 2^{k-1}-1$.
Consider the element $z_{i-2^{k-1},j}$ and observe that
$$
z_{i-2^{k-1},j}\in\gamma_{2^{k}-1}(G)\cap Z
$$
as
$$
i-2^{k-1}+j\ge 2^{k}-1.
$$
Therefore $[z_{i-2^{k-1},j},x^{2^{k-1}}]\in\Delta.$
By Corollary~\ref{corollary p^k}, it follows then that
$$
z_{i,j}z_{i-2^{k-1},j+2^{k-1}}z_{i,j+2^{k-1}}\in\Delta.
$$
Now we have $i>j+2^{k-1}$ and $j+2^{k-1}\ge 2^{k-1}$, so $z_{i,j+2^{k-1}}\in\Delta$.
Next, if $i-2^{k-1}>j+2^{k-1}$, then $z_{i-2^{k-1},j+2^{k-1}}\in\Delta$, and if $i-2^{k-1}\le j+2^{k-1}$, then as  $i-2^{k-1}\ge 2^{k-1}$, we have
$$
z_{i-2^{k-1},j+2^{k-1}}=z_{j+2^{k-1},i-2^{k-1}}^{-1}\in\Delta.
$$
Therefore $z_{i,j}\in\Delta$ and $\gamma_{2^{k}+2^{k-1}-1}(G)\cap Z\le\Delta$.

Finally, for $j\ge 2^{k-1}$, observe that
$$
[c_j,x^{2^{k-1}}]\equiv c_{j+2^{k-2}}^{\,2}c_{j+2^{k-1}}\pmod{\gamma_{2j+2^{k-1}}(G)\cap Z},
$$
and since $\gamma_{2j+2^{k-1}}(G)\cap Z\le\Delta$ and $c_i^{\,2}\in\Delta$ for all $i\ge 2^{k-1}$, we have $c_j\in\Delta$ for all $j\ge 2^k$. We conclude that
$$
T_k\big(\gamma_{2^{k}+2^{k-1}-1}(G)\cap Z\big)\le \Delta\le\Phi_k(G),
$$
as claimed. In particular, we get
$$
\gamma_{2^{k}+2^{k-1}-1}(G)\cap Z\le\Phi_k(G)\cap Z\le \Gamma_k\cap Z.
$$
Now, from Remark~\ref{remark index} we deduce that
\begin{equation}\label{eq:above}
\varliminf_{k\rightarrow\infty}\dfrac{2^{k}+2^{k-1}-1}{\log_2|Z:\Gamma_{k}\cap Z|}= \lim_{k\rightarrow\infty}\dfrac{2^{k}+2^{k-1}-1}{\log_2|Z:\Gamma_{k}\cap Z|}=  0.
\end{equation}
Hence, by Proposition~\ref{proposition path/area}, it only remains to show that $\hdim_G^{\mathcal{F}}(Z)=1$ and that it is given by a proper limit.
This follows easily since, from \cite[Prop.~2.6(3)]{KT} and (\ref{eq:above}), we deduce that
$$
 \lim_{k\rightarrow\infty}\dfrac{\log_2|G:\Phi_k(G)Z|}{\log_2|Z:\Phi_k(G)\cap Z|}=0,
$$
and so, as done in the proof of Theorem $\ref{LD}$, we obtain $\hdim_G^{\mathcal{F}}(Z)=1$.
It further follows that $Z$ has strong Hausdorff dimension with respect to~$\mathcal{F}$.
\end{proof}

\end{document}